\documentclass[reqno,centertags,12pt]{amsart}
\usepackage{amsmath,amsthm,amscd,amssymb,latexsym,mathtools,verbatim}
\usepackage[hidelinks]{hyperref}

\usepackage{graphicx,epsf,cite}

\usepackage[shortlabels]{enumitem}

\usepackage{xcolor}

\usepackage{tikz}
\pgfdeclarelayer{nodelayer}
\pgfdeclarelayer{edgelayer}
\pgfsetlayers{edgelayer,nodelayer,main}
\tikzstyle{arrow}=[draw=black,arrows=-latex]

-\marginparwidth \oddsidemargin -4mm \evensidemargin \oddsidemargin

\newtheorem{theorem}{Theorem}[section]

\newtheorem{lemma}[theorem]{Lemma}

\theoremstyle{definition}
\newtheorem{definition}[theorem]{Definition}

\newcounter{smalllist}

\DeclareMathOperator*{\sgn}{sgn}

\allowdisplaybreaks
\numberwithin{equation}{section}

\newcommand{\abs}[1]{\left\lvert#1\right\rvert}

\newcommand{\norm}[1]{\left\|#1\right\|}

\newcommand{\lb}{\label}

\newcommand{\beq}{\begin{equation}}
\newcommand{\eeq}{\end{equation}}

\newcommand{\bal}{\begin{align}}
\newcommand{\eal}{\end{align}}
\newcommand{\bals}{\begin{align*}}
\newcommand{\eals}{\end{align*}}

\newcommand{\bbR}{{\mathbb{R}}}

\newcommand{\bbZ}{{\mathbb{Z}}}

\newcommand{\bbT}{{\mathbb{T}}}
\newcommand{\bbS}{{\mathbb{S}}}

\begin{document}
\title[Absence of  Splash-like Singularities for g-SQG Patches]
{An Improved Regularity Criterion and \\ Absence of  Splash-like Singularities for
g-SQG Patches}

\author{Junekey Jeon and Andrej Zlato\v{s}}

\address{\noindent Department of Mathematics \\ University of
California San Diego \\ La Jolla, CA 92093 \newline Email: \tt
zlatos@ucsd.edu,
j6jeon@ucsd.edu}



\begin{abstract} 
We prove that splash-like singularities cannot occur for sufficiently regular patch solutions to the generalized surface quasi-geostrophic equation on the plane or half-plane with parameter $\alpha\le \frac 14$.  This includes potential touches of more than two patch boundary segments in the same location, an eventuality that has not been excluded previously 
and presents nontrivial complications  
(in fact, if we do a priori exclude it, then our results extend to all $\alpha\in(0,1)$).  As a corollary, we obtain an improved global regularity criterion for $H^3$ patch solutions when $\alpha\le\frac 14$, namely that finite time singularities cannot occur while the $H^3$ norms of  patch boundaries remain bounded.
\end{abstract}

\maketitle

\section{Introduction} \lb{S1}

The \emph{g-SQG (generalized surface quasi-geostrophic) equation}, is the active scalar PDE
\begin{equation}\label{1.1}
    \partial_{t} \omega + u \cdot \nabla\omega = 0,
\end{equation}
where the scalar $\omega\colon \bbR^{2} \times (0,\infty) \to \bbR$ is advected by the velocity field 
\beq\lb{1.1a}
u\coloneqq \nabla^{\perp}(-\Delta)^{-1+\alpha}\omega.
\eeq
Here $\nabla^{\perp}\coloneqq (-\partial_{x_{2}},\partial_{x_{1}})$ and $\alpha\in\left(0,1\right)$ is a given parameter.  Note that \eqref{1.1} is the vorticity form of the (incompressible) 2D Euler equation when $\alpha=0$, which models the motion of ideal fluids,  with  $u$ the fluid velocity and $\omega:= \nabla^\perp\cdot u$ its vorticity.  When $\alpha=\frac 12$, it is the SQG equation,
which is used in atmospheric science models \cite{Ped} and was first analyzed rigorously by Constantin, Majda, and Tabak \cite{CMT}.  The g-SQG equation with $\alpha\in(0,1)$ is its generalization
and has also been studied in both geophysical and mathematical  literature, including in \cite{PHS, Smith, CIW, CCCGW, KisYaoZla, KRYZ, Gancedo, CorFonManRod, KisLuo}.

Global regularity for (smooth or bounded) solutions has been known  in the Euler case $\alpha=0$ since the works of H\" older \cite{Holder}, Wolibner \cite{Wolibner}, and Yudovich \cite{Yudth}, 
but it is still an open problem for g-SQG with any $\alpha\in(0,1)$.
In this work we consider so-called {\it patch solutions}  to \eqref{1.1}, that is, weak solutions that are linear combinations of characteristic functions of some time-dependent sets $\Omega_n(t)\subseteq\bbR^2$ (often only a single such set/patch  is considered but the extension to multiple sets is typically straightforward).  The main question now is that of global well-posedness for these solutions:  if the boundary of each initial patch $\Omega_n(0)$ is a simple closed curve of some prescribed regularity ($H^k$ or $C^{k,\gamma}$) and these curves are pairwise disjoint, does this setup persist forever or may it cease existing in finite time.  This of course involves not only the required regularity of each $\partial \Omega_n(t)$, but also that they all remain pairwise disjoint simple closed curves.  

Chemin showed the answer to be in the affirmative when $\alpha=0$ \cite{Chemin}, but the question remains open for any $\alpha\in(0,1)$.  Local existence for these models was proved for $\alpha\in(0,\frac 12]$ and $H^3$ patches by Gancedo, who also obtained uniqueness for $\alpha\in(0,\frac 12)$ and those solutions that satisfy a related {\it contour equation}  \cite{Gancedo} (well-posedness for $\alpha=\frac 12$ in a special class of patches was earlier proved by Rodrigo \cite{Rodrigo}).  Local existence was also proved  for $\alpha\in[\frac 12,1)$ and $H^4$ patches by Chae, Constantin, C\' ordoba, Gancedo, and Wu \cite{CCCGW}.
Kiselev, Yao, and Zlato\v s later proved full local well-posedness  for $\alpha\in(0,\frac 12)$ and $H^3$ patches \cite{KisYaoZla} (they also considered the related half-plane case, in which global well-posedness was proved to fail by Kiselev, Yao, Ryzhik, and Zlato\v s \cite{KRYZ}).  In addition,  C\' ordoba, C\' ordoba, and Gancedo achieved this for $\alpha=\frac 12$ and $H^3$ patches  \cite{CorCorGan},  Gancedo and Patel for $\alpha\in(0,\frac 12)$ and $H^2$ patches as well as for $\alpha\in(\frac 12,1)$ and $H^3$ patches, and Gancedo, Nguyen, and Patel for $\alpha=\frac 12$ and $H^{2+\gamma}$ patches  \cite{GanNguPat}.

The singularity-formation mechanism on the half-plane from \cite{KRYZ}, which was motivated by numerical simulations for the 3D Euler equation due to Luo and Hou \cite{LuoHou, LuoHou2}, and by the related proof of double-exponential growth of gradients for smooth solutions to the 2D Euler equation on a bounded domain by Kiselev and \v Sver\' ak \cite{KS}, does not seem to extend to the whole plane case.  It is therefore still unknown whether global well-posedness holds on $\bbR^2$ for any $\alpha\in(0,1)$. Nevertheless, the local well-posedness result in \cite{KisYaoZla} does show that, at least for $\alpha\in(0,\frac 12)$ and $H^3$ patches, finite time singularity can only occur if either a patch boundary loses $H^3$ regularity or a touch happens.  The latter might involve two or more patch boundary segments, which might belong to different patches or to a single patch.

The main result of this paper is that for $\alpha\in(0,\frac 14]$,  a touch cannot occur without the loss of $C^{1,\frac{2\alpha}{1-2\alpha}}$ (and hence also $H^3$) regularity of a patch boundary at the same time (this is also suggested by numerical simulations of C\' ordoba, Fontelos, Mancho, and Rodrigo \cite{CorFonManRod}).
If it did occur and the $C^{1,\gamma}$ norm of the patch boundary would stay uniformly bounded for some $\gamma>0$, the resulting singularity would be called a {\it splash}.  One might think that its existence for the free boundary Euler equation, demonstrated by Castro, C\'ordoba, Fefferman, Gancedo, and G\'omez-Serrano \cite{CCFGG}, and Coutand and Shkoller   \cite{CouShk}, would suggest its possibility for g-SQG patches as well.  But these two cases are very different: the converging boundary segments are separated by vacuum in the free boundary case, while for \eqref{1.1} they are  separated by the (incompressible) fluid medium, which must be ``squeezed out'' of the region between them before a touch can occur.  

One might also think that impossibility of general splash singularities was already proved by Gancedo and Strain for the SQG case $\alpha=\frac 12$ and smooth patches  \cite{GanStr}, who showed that a touch of two patch boundary segments (which we call a {\it simple splash})  is indeed impossible at any specific location without a loss of boundary smoothness   (their argument extends to all $\alpha\in(0,\frac 12)$).  However, they proved this assuming that no singularity occurs elsewhere, and the result also does not exclude simultaneous touches of three or more boundary segments.  Crucially, their proof does not extend to this case either.  In it, they place the two segments in a coordinate system in which both are close to horizontal, and use the fact that normal vectors at two points that minimize the vertical distance of the two boundary segments (at any given time) are automatically parallel. This causes important cancellations in the integral evaluating the approach velocity of the two points, which bound this velocity by a multiple of the product of the distance of the two points and the log of this distance. Gr\" onwall's inequality then yields at most double-exponential-in-time approach rate of the two segments.

We can even obtain a simple exponential bound for $C^{2,\gamma}$ patches with $\gamma>0$ by instead minimizing the distance (rather than vertical distance) of the two boundary segments, in which case the normals at the closest points both lie on the line connecting these points. The resulting computation then bounds the approach velocity by only a multiple of the distance, and it even extends to all $\alpha\in(0,1)$ with appropriate $\gamma$ (see Subsection \ref{S2.5} below). 

However, when a third boundary segment is present nearby, its normal vector at the point where it intersects  the above line 
need not lie on that line, which significantly compromises the  cancellations involved.  One then needs to obtain very precise bounds on the resulting errors in this case, which we will achieve by using the uniform $C^{1,\frac{2\alpha}{1-2\alpha}}$ bound on the patch boundary to estimate the angle between this normal vector and the line, in terms of the distance of the third segment from the two closest points on the first two segments.  When this distance is small, the error will be controlled because the angle must be small;  this control worsens when the distance is larger, but then the effect of the third segment on the two points decreases as well.  This will yield the needed bound on the approach velocity of the closest points, and this estimate will even extend to the case of arbitrarily many boundary segments folded on top of each other and attempting to create a {\it complex splash} singularity.

As a result, we will obtain an improved regularity criterion for $H^3$ patch solutions to \eqref{1.1}, requiring only a uniform bound on the $C^{1,\frac{2\alpha}{1-2\alpha}}$ norm of the patch boundaries.  Nevertheless, this approach only works when $\alpha\in(0,\frac 14]$, and the obtained estimates are insufficient for larger $\alpha$ (specifically, Lemma \ref{L2.5} below).  The reason for this is not just technical, and simply assuming higher boundary regularity will not suffice to overcome the new complications involved.  We believe that a different (dynamical) approach will be needed for $\alpha>\frac 14$ (if the result extends to this range at all), which likely makes it a very difficult problem.

Let us now state rigorously the definition of patch solutions to \eqref{1.1} from \cite{KisYaoZla} (which even allows patches to be nested), and our main result.  Below we let $\mathbb{T}\coloneqq \bbR/2\pi\bbZ$.

\begin{definition} \lb{D1.1}
    Let $\Omega\subseteq\bbR^{2}$ be a bounded open set whose boundary $\partial\Omega$ is a simple closed $C^{1}$ curve with arc-length $\abs{\partial\Omega}$. We call \emph{a constant-speed parameterization of $\partial\Omega$} any counterclockwise parameterization $z\colon\mathbb{T}\to \bbR^{2}$ of $\partial\Omega$ with $\abs{z'}\equiv\frac{\abs{\partial\Omega}}{2\pi}$ on $\mathbb{T}$ (these are all translations of each other), and we define $\norm{\Omega}_{C^{k,\gamma}}\coloneqq \norm{z}_{C^{k,\gamma}(\bbT)}$  and $\norm{\Omega}_{H^{k}}\coloneqq\norm{z}_{H^{k}}$ for $(k,\gamma)\in\mathbb{N}_{0}\times[0,1]$.
\end{definition}

Next we note that when $\alpha\in(0,\frac 12)$,  the velocity $u$ from \eqref{1.1a} satisfies the explicit formula
\begin{equation}\label{1.2}
    u(x,t) \coloneqq  c_\alpha \int_{\bbR^{2}} \frac{(x-y)^{\perp}}{\abs{x-y}^{2+2\alpha}} \, \omega(y,t) \,dy
\end{equation}
for bounded $\omega$,
with $v^{\perp}\coloneqq (-v_{2},v_{1})$ and $c_\alpha>0$ an appropriate constant (see Subsection~\ref{S2.5} below for the necessary adjustments when $\alpha\in[\frac 12,1)$).
For any $\Gamma\subseteq\mathbb{R}^{2}$, vector field $v\colon\Gamma\to\mathbb{R}^{2}$, and $h\in\mathbb{R}$, we let 
the set to which $\Gamma$ is advected by $v$ in time $h$ be
\[
    X_{v}^{h}[\Gamma]\coloneqq \left\{x+hv(x) \,\big|\, x\in\Gamma\right\}.
\]

\begin{definition} \lb{D1.2}
Let $\theta_1,\dots,\theta_N\in\bbR\setminus\{0\}$, and for each $t\in[0,T)$, let $\Omega_1(t),\dots,\Omega_N(t)\subseteq \bbR^2$ be bounded open sets whose boundaries 
are pairwise disjoint simple closed curves such that each $\partial \Omega_n(t)$ is also continuous in $t\in[0,T)$ with respect to Hausdorff distance $d_H$ of sets.  Denote $\partial\Omega(t):=\bigcup_{n=1}^N \partial\Omega_n(t)$ and $\norm{\Omega(t)}_{Y}:=\sum_{n=1}^N \norm{\Omega_n(t)}_{Y}$ for $Y\in\{C^{k,\gamma}, H^k\}$, and let
\begin{equation} \label{1.55}
\omega(\cdot,t) := \sum_{n=1}^N \theta_n \chi_{\Omega_n(t)}.
\end{equation}
    If  for each $t\in(0,T)$ we have
    \begin{equation} \lb{1.3a}
        \lim_{h\to 0} \frac{d_{H}\left( \partial \Omega(t+h),
        X_{u(\cdot,t)}^{h}[\partial\Omega(t)]\right)}{h} = 0,
    \end{equation}
    with $u$ from \eqref{1.2}, then $\omega$ is a \emph{a patch solution} to \eqref{1.1}-\eqref{1.1a} on the time interval $[0,T)$. If we also have $\sup_{t\in[0,T']}\norm{\Omega(t)}_{Y}<\infty$ for some $Y\in\{C^{k,\gamma}, H^k\}$ and each $T'\in(0,T)$, then $\omega$ is a \emph{$Y$ patch solution}  to \eqref{1.1}-\eqref{1.1a} on $[0,T)$.
\end{definition}

While \eqref{1.3a} is stated for each single time $t$ (akin to the definition of strong or classical solutions to a PDE), it agrees with the usual flow-map based definition of solutions to the 2D Euler equation (see the remarks after Definition 1.2 in  \cite{KisYaoZla}).  Since $u$ is only H\" older continuous at the patch boundaries when $\alpha>0$ (and hence the flow map may not be unique), this definition is more appropriate in the g-SQG case.  
    
Theorem 1.5 in \cite{KisYaoZla} shows  that for any $\theta_1,\dots,\theta_N\in\bbR\setminus\{0\}$ and any bounded open sets $\Omega_1(0),\dots,\Omega_N(0)\subseteq \bbR^2$ 
whose boundaries are pairwise disjoint simple closed $H^3$ curves,
there is a time $T\in(0,\infty]$ such that a unique $H^3$ patch solution $\omega= \sum_{n=1}^N \theta_n \chi_{\Omega_n(\cdot)}$ to \eqref{1.1}-\eqref{1.1a} 
exists on $[0,T)$.  And if the maximal such $T$ is finite, then either $\sup_{t\in[0,T)}\norm{\Omega(t)}_{H^{3}}=\infty$ or
\beq\lb{1.3b}
    \sup_{t\in[0,T)} \,\sup_{(n,\xi),(j,\eta)\in \{1,\dots,N\} \times \bbT \,\&\, (n,\xi)\neq(j,\eta)} \frac{|n-j| + \abs{\xi-\eta}}{\abs{z_n(\xi,t) - z_j(\eta,t)}} = \infty,
\eeq
where $z_n(\cdot, t)$ is a constant-speed parametrization of $\partial\Omega_n(t)$ and $|\xi-\eta|$ is  distance on $\bbT$.
Note that if \eqref{1.3b} holds with $n=j$, which means that the {\it arc-chord ratio} for some $\Omega_n(\cdot)$ becomes unbounded as $t\to T$, then this must be realized by a touch of ``distinct'' segments (which we call {\it folds}) of $\partial\Omega_n(\cdot)$ whenever $\sup_{t\in[0,T)}\norm{\Omega(t)}_{C^{1,\gamma}}<\infty$ for some $\gamma>0$.  Indeed,  since $|\partial_\xi z_n(\xi,t)|$ is then uniformly-in-$(n,\xi,t)$ bounded below by a positive constant (see \eqref{2.1a}), it follows that 
the fraction in \eqref{1.3b} is uniformly bounded above when $n=j$ and $|\xi-\eta|$ is small enough. 
Therefore \eqref{1.3b} is the correct definition of a {\it splash-like singularity}   at time $T$ (i.e., a touch of either boundaries of distinct patches or folds of the same patch boundary,  including both simple and complex splashes) when $\sup_{t\in[0,T)}\norm{\Omega(t)}_{C^{1,\gamma}}<\infty$ for some $\gamma>0$.

The following theorem is now our main result.

%

\begin{theorem}\lb{T1.3}
    If $\alpha\in(0,\frac{1}{4}]$ and a $C^{1,\frac{2\alpha}{1-2\alpha}}$ patch solution to \eqref{1.1}-\eqref{1.1a} 
    on the time interval $[0,T)$ with $T<\infty$ satisfies $\sup_{t\in[0,T)}\norm{\Omega(t)}_{C^{1,\frac{2\alpha}{1-2\alpha}}}<\infty$, then \eqref{1.3b} fails   (so no splash-like singularity can occur).  In particular, if  the maximal time $T$ of existence of an $H^3$ patch solution from Theorem~1.5 in \cite{KisYaoZla} is finite (and $\alpha\in(0,\frac{1}{4}]$), then $\sup_{t\in[0,T)}\norm{\Omega(t)}_{H^{3}}=\infty$.    
\end{theorem}

{\it Remarks.} 
1.  Our proof shows that the left-hand side of \eqref{1.3b} with $\sup_{t\in[0,T)}$ removed can grow at most exponentially in time (up to time $T$) if $\sup_{t\in[0,T)}\norm{\Omega(t)}_{C^{1,\frac{2\alpha}{1-2\alpha}}}<\infty$.  Hence boundaries of distinct patches, as well as folds of the same patch boundary, can only approach each other exponentially quickly in this case.
\smallskip 

2.  While we do not know whether this result extends to some $\alpha>\frac 14$,  in Subsection~\ref{S2.5} below we provide an extension to all $\alpha\in\left(0,1 \right)$ when one a priori requires that only simple splashes can occur (i.e., no more than two segments of $\partial\Omega$ are allowed to touch in the same location) and $\sup_{t\in[0,T)}\norm{\Omega(t)}_{C^{k,\gamma}}<\infty$ holds for $k=1$ and some $\gamma\ge 2\alpha$ (when $\alpha\in(0,\frac 12]$), or for $k=2$ and some $\gamma\ge 2\alpha-1$ (when $\alpha\in[\frac 12,1)$).   
The obtained bound on the approach rate of two patches/folds is now double-exponential when $\gamma$ is equal to the minimal value above ($2\alpha$ or $2\alpha-1$) and exponential otherwise.  We note that
when the potential simple splash is assumed to have  a predetermined location and development of singularities elsewhere is a priori excluded,
then this was also proved for $\alpha= \frac 12$ in \cite{GanStr} (for smooth patches and with a double-exponential bound on the approach rate), and  by Kiselev and Luo for all $\alpha\in(0,1)$ in \cite{KisLuo} (this work was done contemporaneously with and independently of ours).\smallskip


Finally, here is an extension to the half-plane (see Section \ref{S3} for the relevant adjustments).

\begin{theorem}\lb{T1.4}
Theorem \ref{T1.3} extends to patch solutions on the half-plane, with the second claim involving $H^3$ patch solutions from \cite[Theorem~1.4]{KisYaoZla} and $\alpha\in(0,\frac 1{24})$,
 or $H^2$ patch solutions from \cite[Theorem~1.1]{GanPat} and $\alpha\in(0,\frac 1{6})$.
\end{theorem}

It was proved in  \cite{KRYZ} that for any $\alpha\in(0,\frac 1{24})$,  there are $H^3$ patch solutions on the half-plane that become singular in finite time.  For $\alpha\in(0,\frac 1{6})$ and $H^2$ patch solutions this was proved in \cite{GanPat}. Theorem \ref{T1.4} shows that this cannot happen only via a splash-like singularity and always involves blow-up of  their $H^3$ resp.~$H^2$ norms.

\medskip

\textbf{Acknowledgements}.
JJ acknowledges partial support by  NSF grant DMS-1900943.  AZ acknowledges partial support by  NSF grant DMS-1900943 and by a Simons Fellowship.  

\section{Proof of Theorem~\ref{T1.3}} \lb{S2}


\subsection{The Single Patch Case} \lb{S2.1}

For the sake of notational simplicity, let us first consider the case of a single patch on which $\omega\equiv 1$, that is, $\omega(\cdot,t)=\chi_{\Omega(t)}$.  Then \eqref{1.2} becomes 
    \begin{equation}\lb{1.4}
        u(x,t) \coloneqq \int_{\Omega(t)} \frac{(x-y)^{\perp}}{\abs{x-y}^{2+2\alpha}} \, dy,
    \end{equation}
after rescaling \eqref{1.1} in time by $c_\alpha$ (which we do in order to remove the constant). 

We will not assume $\alpha\leq\frac{1}{4}$ until it is needed, so that it is clear where this hypothesis enters in our argument.   We will therefore consider a $C^{1,\gamma}$ patch solution with any $\gamma\in(0,1]$ below. 
If now $z(\cdot,t)$ is any constant-speed parameterization of $\partial\Omega(t)$ for $t\in[0,T)$, we assume that
\beq\lb{2.0}
M\coloneqq\sup_{t\in[0,T)}\norm{z(\cdot,t)}_{C^{1,\gamma}}<\infty.
\eeq
We now want to show that this implies
\begin{equation}\lb{2.1}
    \sup_{t\in[0,T)} \,\sup_{\xi,\eta\in\bbT \,\&\, \xi\neq\eta} \frac{\abs{\xi-\eta}}{\abs{z(\xi,t) - z(\eta,t)}} < \infty.
\end{equation}

Since a $C^1$ patch solution is also a weak solution to \eqref{1.1}-\eqref{1.1a} with  $\abs{\Omega(t)}$ being conserved (see Remark 3 after Definition 1.2 in \cite{KisYaoZla}), the isoperimetric inequality shows that 
\beq\lb{2.1a}
M'\coloneqq   \inf_{(\xi,t)\in \bbT\times [0,T)} \abs{\partial_{\xi}z(\xi,t)}>0.
\eeq
Now for any $t\in[0,T)$ and $\xi,\eta\in\bbT$, there are  $\xi_{1},\xi_{2}\in\bbT$ between $\xi$ and $\eta$ such that
\[
    \abs{z(\xi,t) - z(\eta,t)} = |\xi-\eta|\abs{(\partial_{\xi}z_{1}(\xi_{1},t), \partial_{\xi}z_{2}(\xi_{2},t))} 
    \geq |\xi-\eta| \left( \abs{\partial_{\xi}z(\xi ,t)} - 2M |\xi-\eta|^\gamma  \right) .
\]
Hence if we let $\delta\coloneqq(M'/4M)^{1/\gamma}$, then 
\beq\lb{2.2a}
 \abs{z(\xi,t) - z(\eta,t)}\ge \frac{M'}2|\xi-\eta|
 \eeq
  whenever $|\xi-\eta|\le\delta$.  It follows that to conclude \eqref{2.1}, it suffices to show
\begin{equation}\lb{2.2}
    \inf_{t\in[0,T)} \, \min_{\xi,\eta \in \bbT \,\&\, \abs{\xi-\eta}\geq\delta} \abs{z(\xi,t) - z(\eta,t)} >0.
\end{equation}


We therefore let
\beq\lb{2.2b}
    m(t)\coloneqq \min_{\xi,\eta\in\bbT \,\&\, \abs{\xi-\eta}\geq\delta}\abs{z(\xi,t) - z(\eta,t)} \ge 0,
\eeq
and let $\xi_{t},\eta_{t}\in\mathbb{T}$ be such that $ \abs{z(\xi_{t},t) - z(\eta_{t},t)} = m(t)$.   If \eqref{2.2} fails, then clearly for all $t<T$ close enough to $T$ we have $m(t)<\frac{M'}2\delta$, which shows that $\abs{\xi_{t}-\eta_{t}}>\delta$ for these $t$ because \eqref{2.2a} holds.  It suffices to consider only such $t$.
Then, following an argument in \cite{ConsEsch}, one can easily see that $m(t)$ is locally Lipschitz (and so differentiable at almost all such $t$) and we have
\begin{equation}\lb{2.3}
    m'(t) = \frac{z(\xi_{t},t) - z(\eta_{t},t)}{m(t)}
    \cdot (u(z(\xi_{t},t),t) - u(z(\eta_{t},t),t))
\end{equation}
for almost every such $t$. Hence Gr\"{o}nwall's inequality shows that it suffices to prove
\begin{equation}\lb{2.4}
   -(u(z(\xi_{t},t),t) - u(z(\eta_{t},t),t)) \cdot n_t \le C m(t)
\end{equation}
for some $t$-independent $C<\infty $ and all $t$ such that $m(t)\in(0,\frac{M'\delta}2)$, where $n_t\coloneqq \frac{z(\xi_{t},t) - z(\eta_{t},t)}{m(t)}$ is the unit vector in the direction $z(\xi_{t},t) - z(\eta_{t},t)$.  Of course, the definition of $\xi_t,\eta_t$ shows that $n_t$ is also normal to $\partial\Omega(t)$ at both $z(\xi_{t},t)$ and $z(\eta_{t},t)$.


Since \eqref{2.4} only involves quantities at a single time, we will now assume that $t$ is close to $T$ and  drop the dependence of $\Omega,z,u,m$ on $t$ from our  notation.  The above also shows that after a translation and rotation we can assume that:
\begin{enumerate}
    \item $z(\eta_{t}) = (0,0)$ and $z(\xi_{t}) = (0,m)$, with $m \in(0,\frac{M'\delta}2)$;
    \item $\partial_{\xi}z(\xi_{t}), \partial_{\xi}z(\eta_{t})\perp(0,1)$.
\end{enumerate}
We will do so, and then \eqref{2.4} becomes just
\begin{equation}\lb{2.5}
    u_{2}(0,0) - u_{2}(0,m) \le C m.
\end{equation}
We will prove this in the next three subsections.  We note that all constants below may depend on $\alpha,\gamma,M,M'$ (recall that $\delta$ also depends on these), but will be independent of $m,t$.

\subsection{Some Geometric Lemmas} \lb{S2.2}

We first state some geometric lemmas that will be used throughout. The first of these is a trivial consequence of $C^{1,\gamma}$-regularity of $\partial\Omega$, which says that near any $z(\xi)$, the curve $z$ is the graph of some function $f:\bbR\to\bbR$ defined with respect to the coordinate system centered at $z(\xi)$ and with the horizontal axis not too far from $\partial_{\xi}z(\xi)$.

\begin{lemma}\lb{L2.1}
    There are  $A\ge 1$ and $R_{0}>0$ such that for any $\xi\in\mathbb{T}$ and any $v\in\bbS^1$ with $| \partial_{\xi}z(\xi) \cdot v|\ge \frac 12|\partial_{\xi}z(\xi)|$, there is $f\colon[-R_{0},R_{0}]\to\mathbb{R}$ with $\norm{f}_{C^{1,\gamma}}\leq A$ such that
    \[
        \left\{z(\xi) + hv
        + f(h) v^\perp
       \,\,\big|\,\, h\in [-R_{1},R_{1}]\right\}
       = z \left( \left[\xi-\xi_1,\xi+\xi_2 \right] \right)
    \]
for each $R_1\in[0,R_0]$ and some $\xi_1,\xi_2\in[\frac{R_1}{M}, \frac{3R_1}{M'}]$.  Then $f(0)=0$ and $f'(0)=\frac{\partial_{\xi}z(\xi)\cdot v^\perp}{\partial_{\xi}z(\xi)\cdot v}$.
\end{lemma}

The next lemma shows that when two folds of $\partial \Omega$ are close to each other, the angles between their tangent lines are controlled by their distance.

\begin{lemma}\lb{L2.2}
    There are $B,R>0$ such that for any $\xi,\eta\in\mathbb{T}$ with 
    $\abs{z(\xi)-z(\eta)}\leq R$ we have
    \beq\lb{2.5a}
        \abs{\tan\theta} \leq B\abs{z(\xi)-z(\eta)}^{\frac{\gamma}{1+\gamma}},
    \eeq
 where $\theta$ is the angle between $\partial_{\xi}z(\xi)$ to $\partial_{\xi}z(\eta)$.
\end{lemma}

\begin{proof}
Let $A,R_0$ be from Lemma \ref{L2.1}. First note that it suffices to prove 
    \begin{equation}\lb{2.5b}
        \abs{\sin\theta} \leq B\abs{z(\xi)-z(\eta)}^{\frac{\gamma}{1+\gamma}}
    \end{equation}
    instead of \eqref{2.5a}.  Indeed, we then only need to replace $R$ by $\min\{R, (2B)^{-\frac{1+\gamma}{\gamma}}\}$ (which yields $\abs{\cos\theta}\geq\frac{1}{2}$) and then double $B$.
        
Take  $C:= 9A$, and let $R:=\min\{\frac 12C^{-\frac{2+2\gamma}{\gamma}},(\frac{R_0}{3C})^{2}\}$ and  $B:=3C^2$.  Without loss assume that $z(\eta)=0$ and $\frac{\partial_{\xi}z(\eta)}{\abs{\partial_{\xi}z(\eta)}}=(1,0)$, and then let $r:=\abs{z(\xi)}\le R$ and $r':=C r^{\frac{1}{1+\gamma}}\le CR^{\frac 12}\le\frac{R_0}3$.
Then Lemma \ref{L2.1} with $v:=(1,0)$  shows  that $z$ near the origin is a curve connecting the vertical sides of the rectangle $Q:=[-3r',3r']\times[-C^3r,C^3r]$ because
\[
A(3r')^\gamma (3r') \le 9AC^2r \le C^3r
\]
(note that the definition of $R$ shows that $C^3r< r'$, so the vertical sides are the shorter ones).

Apply the same argument with $v:=\frac{\partial_{\xi}z(\xi)}{\abs{\partial_{\xi}z(\xi)}}$ and the rectangle $Q'$ centered at $z(\xi)$ whose longer axis connects the points $z(\xi)\pm v r'$ and whose shorter sides have again length $2C^3r$.  It shows that $z$ near $z(\xi)$ is a curve connecting the shorter sides of $Q'$.  If \eqref{2.5b} is violated, then one of these sides lies fully in $(-3r',3r')\times(C^3r,\infty)$ and the other in $(-3r',3r')\times(-\infty,-C^3r)$ (see Figure \ref{F.2.9}) because $r+r'+C^3r<3r'$ and
\[
r'\sin\theta > BCr \ge 3C^3r > 2C^3r+r.
\]  
But this means that the two curves must intersect, a contradiction with our assumption that no touch has occurred before time $T$.
\end{proof}

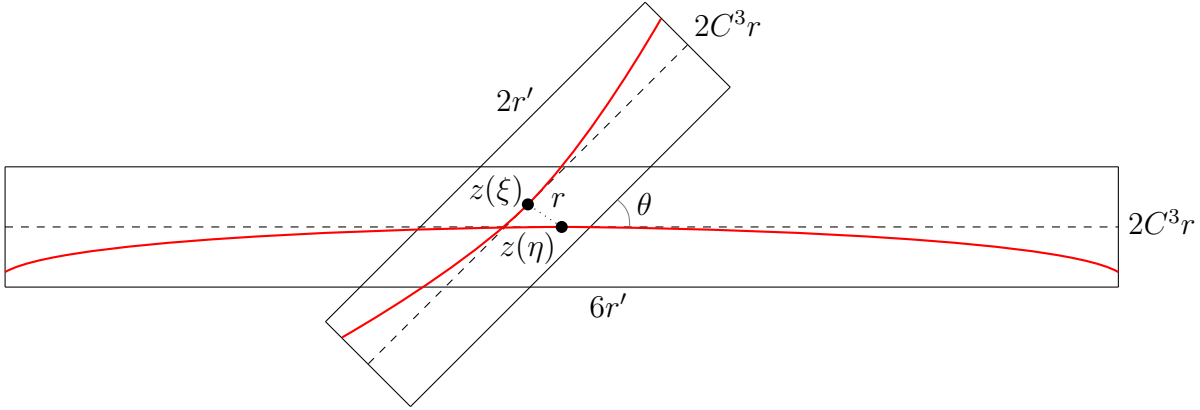
\begin{figure} 
    \pgfdeclarelayer{drawing layer}
    \pgfdeclarelayer{label layer}
    \pgfsetlayers{drawing layer,label layer}
    \begin{tikzpicture}
        \node (eta) at (0, 0) {};
        \node (eta_tangent_l) at (-7.4, 0) {};
        \node (eta_tangent_r) at (7.4, 0) {};
        \node (eta_curve_l) at (-7.4, -0.6) {};
        \node (eta_curve_r) at (7.4, -0.6) {};
        \node (eta_rect_tr) at (7.4, 0.8) {};
        \node (eta_rect_tl) at (-7.4, 0.8) {};
        \node (eta_rect_bl) at (-7.4, -0.8) {};
        \node (eta_rect_br) at (7.4, -0.8) {};
        \node (xi) at (-0.45, 0.3) {};
        \node (xi_tangent_r) at (1.67, 2.42) {};
        \node (xi_tangent_l) at (-2.57, -1.82) {};
        \node (xi_curve_l) at (-2.92, -1.47) {};
        \node (xi_curve_r) at (1.32, 2.77) {};
        \node (xi_rect_tr) at (1.11, 2.99) {};
        \node (xi_rect_tl) at (-3.14, -1.26) {};
        \node (xi_rect_bl) at (-2.01, -2.39) {};
        \node (xi_rect_br) at (2.24, 1.86) {};
        \node (theta_from) at (0.9, 0) {};
        \node (theta_to) at (0.74, 0.36) {};

        \begin{pgfonlayer}{drawing layer}
            \draw (eta_rect_tr.center) to (eta_rect_tl.center);
            \draw (eta_rect_tl.center) to (eta_rect_bl.center);
            \draw (eta_rect_bl.center) to (eta_rect_br.center);
            \draw (eta_rect_br.center) to (eta_rect_tr.center);
            \draw (xi_rect_tr.center) to (xi_rect_tl.center);
            \draw (xi_rect_tl.center) to (xi_rect_bl.center);
            \draw (xi_rect_bl.center) to (xi_rect_br.center);
            \draw (xi_rect_br.center) to (xi_rect_tr.center);
            \draw[dashed] (eta_tangent_l.center) to (eta_tangent_r.center);
            \draw[dashed] (xi_tangent_l.center) to (xi_tangent_r.center);
            \draw[dotted] (eta.center) to (xi.center);
            \draw [red, thick, in=30, out=180, looseness=0.35] (eta.center) to (eta_curve_l.center);
            \draw [red, thick, in=150, out=0, looseness=0.35] (eta.center) to (eta_curve_r.center);
            \draw [red, thick, in=30, out=225, looseness=0.75] (xi.center) to (xi_curve_l.center);
            \draw [red, thick, in=-120, out=45, looseness=0.75] (xi.center) to (xi_curve_r.center);
            \draw [thin, gray, bend right, looseness=1] (theta_from.center) to (theta_to.center);

            \fill (eta) circle(0.08);
            \fill (xi) circle(0.08);
        \end{pgfonlayer}

        \begin{pgfonlayer}{label layer}
            \node (eta_label) at (-0.45, -0.28) {$z(\eta)$};
            \node (xi_label) at (-0.87, 0.48) {$z(\xi)$};
            \node (r_label) at (-0.04, 0.35) {$r$};
            \node (theta_label) at (1.1, 0.3) {$\theta$};
            \node (eta_hori_label) at (0.62, -1.04) {$6r'$};
            \node (eta_vert_label) at (7.98, 0.07) {$2C^{3}r$};
            \node (xi_vert_label) at (2.21, 2.71) {$2C^{3}r$};
            \node (xi_hori_label) at (-0.62, 1.7) {$2r'$};
        \end{pgfonlayer}
    \end{tikzpicture}
    \caption{The curve $z$ inside the rectangles $Q$ and $Q'$.} \lb{F.2.9}
\end{figure}

We can now combine Lemmas~\ref{L2.1} and \ref{L2.2} to obtain the following constraint on the geometry of $\partial\Omega$ near the origin.

\begin{lemma}\lb{L2.3}
In the setting of (1) and (2), there are $A,B,R>0$ with $B(3R)^{\frac{\gamma}{1+\gamma}} \le\frac 12$ and $M(4R)^\gamma\le (M')^{1+\gamma}$ such that for any $\xi\in\mathbb{T}$ with $z(\xi)\in[-R,R]\times[-2R,2R]$, there is  $f\colon[-R,R]\to\mathbb{R}$ with
    \[
        \norm{f}_{C^{1,\gamma}}\leq A
        \quad\textrm{and}\quad
        \abs{f'(z_{1}(\xi))} \leq B\abs{z(\xi)}^{\frac{\gamma}{1+\gamma}}
    \]
    such that the graph of $f$ is a segment of the curve $z$ around $z(\xi)$. In particular,
    \[
        \abs{f(h)- z_{2}(\xi) - f'(z_{1}(\xi))(h - z_{1}(\xi))}
        \leq A\abs{h - z_{1}(\xi)}^{1+\gamma}
    \]
for all $h\in[-R,R]$. And if $\abs{f(h')}> 2R$ for some $h'\in[-R,R]$, then $\abs{f(h)} > R$ for all $h\in[-R,R]$.
\end{lemma}

\begin{proof}
The first statement is an immediate consequence of Lemmas~\ref{L2.1} and \ref{L2.2}, with $A$ from Lemma \ref{L2.1}, and $B$ from Lemma~\ref{L2.2}, and $R$ being the minimum of one third of $R$ from Lemma~\ref{L2.2} and $\min\{ \frac 13 (2B)^{-\frac{1+\gamma}{\gamma}}, \frac 14 (M')^{\frac{1+\gamma}{\gamma}} M^{-\frac{1}{\gamma}}\}$ (Lemma \ref{L2.1} is applied with $v\coloneqq(1,0)$ and $z_2(\xi)$ is added to the obtained $f$).  The second statement is its immediate consequence, while the third holds by $B(3R)^{\frac{\gamma}{1+\gamma}} \le\frac 12$.
\end{proof}

The third claim shows that any connected component of $\partial\Omega\cap ([-R,R]\times[-2R,2R])$ that intersects $[-R,R]^2$ is a graph of a function $f:[-R,R]\to [-2R,2R]$ that satisfies the lemma.
Note also that since the arc-length of any such component is at least $2R$, and the arc-length of $\partial\Omega$ is uniformly bounded above because so is $\|\partial\Omega\|_{C^{1}}$, it follows that the number of such components is bounded above by some constant  $K$.  That is, if we assume \eqref{2.0}, only a finite number of folds of $\partial\Omega$ might potentially create a single touch (splash) at time $T$ (we will show below that this is in fact not possible when $\alpha\le \frac 14$).

\subsection{Reduction to Regions Near Individual Boundary Segments} \lb{S2.3}

Take $A,B,R$ from Lemma~\ref{L2.3} and $K$ above.
From \eqref{1.4} we see that the left-hand side of \eqref{2.5} is the sum of the terms
\begin{align*}
    I &\coloneqq \int_{\Omega\cap [-R,R]^2} \left( \frac{y_{1}}{\abs{y}^{2+2\alpha}}
    - \frac{y_{1}}{\abs{y - (0,m)}^{2+2\alpha}} \right) dy, \\
    I' &\coloneqq \int_{\Omega\setminus [-R,R]^2} \left( \frac{y_{1}}{\abs{y}^{2+2\alpha}}
    - \frac{y_{1}}{\abs{y - (0,m)}^{2+2\alpha}} \right) dy.
\end{align*}
To prove \eqref{2.5}, it clearly suffices to assume that $m\leq\frac{R}{2}$, in which case clearly $|I'|\le Cm$ for some constant $C$.  Hence we only need to show that $I\le Cm$.

Assume that $f_{1}, \dots ,f_{k}\colon[-R,R]\to [-2R,2R]$ are distinct functions whose graphs are all the connected components of $\partial\Omega$ from the paragraph after Lemma \ref{L2.3}  (so $k\le K$), and order them so that $f_{1}(0)<\dots<f_{k}(0)$. Let $g_{i}\coloneqq \sgn(f_i) \min\{|f_{i}|,R\}$, so that
\begin{equation}\lb{2.6}
    I \le \sum_{i=1}^{k+1}
   \left| \int_{-R}^{R}\int_{g_{i-1}(h)}^{g_{i}(h)}
    \left( \frac{h}{(h^{2}+v^{2})^{1+\alpha}} - \frac{h}{(h^{2} + (v-m)^{2})^{1+\alpha}} \right) dv\,dh \right|,
\end{equation}
where $g_{0}\equiv -R$ and $g_{k+1}\equiv R$.  Since the integrand is odd in $h$,
its integral on any region symmetric with respect to the vertical axis is zero. 
Since  $[-R,R]\times [g_{i-1}(0),g_{i}(0)]$ is such a region we can replace the integral $\int_{g_{i-1}(h)}^{g_{i}(h)}$ in \eqref{2.6} by the sum of integrals $\int_{g_{i-1}(h)}^{g_{i-1}(0)}$ and $\int_{g_{i}(0)}^{g_{i}(h)}$ (with the same integrand).
We therefore obtain $|I|\le 2 \sum_{i=1}^{k} |I_{i}|$, where
\[
    I_{i} \coloneqq \int_{-R}^{R}\int_{g_{i}(0)}^{g_{i}(h)}
   \left( \frac{h}{(h^{2}+v^{2})^{1+\alpha}} - \frac{h}{(h^{2} + (v-m)^{2})^{1+\alpha}}\,dv \right) dh.
\]
Hence, we are left with showing $|I_{i}|\le C m$ for each $i$. When doing this, we can just assume that $g_{i}=f_{i}$ because the error we incur by this involves only integration over $\Omega\setminus [-R,R]^2$, and therefore is no more than
$Cm$ (similarly to $I'$).

\subsection{Estimating the Individual Integrals} \lb{S2.4}

We thus consider any $f\colon[-R,R]\to [-2R,2R]$ whose graph is a segment of the curve $z$ passing through a point in $[-R,R]^2$, let
\beq \lb{2.7a}
    J \coloneqq \int_{-R}^{R}\int_{f(0)}^{f(h)}
    \left( \frac{h}{(h^{2}+v^{2})^{1+\alpha}} - \frac{h}{(h^{2} + (v-m)^{2})^{1+\alpha}} \right) dv\,dh,
\eeq
and need to show that $\abs{J}\le Cm$. We further divide this integral into two pieces:
\begin{align*}
    J_{1} & \coloneqq \int_{-R}^{R}\int_{f(0)}^{f(0) + hf'(0)}
    \left( \frac{h}{(h^{2}+v^{2})^{1+\alpha}} - \frac{h}{(h^{2} + (v-m)^{2})^{1+\alpha}} \right) dv\,dh, \\
    J_{2} & \coloneqq \int_{-R}^{R}\int_{f(0) + hf'(0)}^{f(h)}
    \left( \frac{h}{(h^{2}+v^{2})^{1+\alpha}} - \frac{h}{(h^{2} + (v-m)^{2})^{1+\alpha}} \right) dv\,dh,
\end{align*}
and estimate $J_{2}$ first.

\begin{lemma} \lb{L2.4}
    We have $\abs{J_{2}}\le Cm$ when $\gamma>2\alpha$ and $\abs{J_{2}}\le Cm(1+\ln_- m)$ when  $\gamma=2\alpha$, for some constant $C$.
\end{lemma}

\begin{proof}
    By Lemma~\ref{L2.3}, we have
    \begin{align*}
        \abs{J_{2}} &\leq \int_{-R}^{R}
        \int_{f(0)+hf'(0) -A\abs{h}^{1+\gamma}}^{f(0)+hf'(0)+A\abs{h}^{1+\gamma}}
        \abs{ \frac{h}{(h^{2}+v^{2})^{1+\alpha}} - \frac{h}{(h^{2} + (v-m)^{2})^{1+\alpha}} }
        \,dv\,dh.
    \end{align*}
    Mean value theorem yields
    \begin{equation*}
        \abs{ \frac{h}{(h^{2}+v^{2})^{1+\alpha}} - \frac{h}{(h^{2} + (v-m)^{2})^{1+\alpha}} }
        = \frac{(2+2\alpha) \abs{h} (h^{2}+\bar{v}^{2})^{\alpha} \abs{\bar{v}} }
        {(h^{2}+v^{2})^{1+\alpha}(h^{2}+(v-m)^{2})^{1+\alpha}} \, m
    \end{equation*}
    for some $\bar{v}\in[v-m,v]$. Hence
    \[
        \abs{ \frac{h}{(h^{2}+v^{2})^{1+\alpha}} - \frac{h}{(h^{2} + (v-m)^{2})^{1+\alpha}} }
       \le \frac{ 3m }{\abs{h}^{2+2\alpha}},
    \]
    and if $V:=\max\{\abs{v},\abs{v-m})\} \ge \abs{h}$, then
    we also have
    \[
        \abs{ \frac{h}{(h^{2}+v^{2})^{1+\alpha}} - \frac{h}{(h^{2} + (v-m)^{2})^{1+\alpha}} }
        \le \frac{3m\abs{h}2^\alpha V^{1+2\alpha}}
        {\abs{h}^{2+2\alpha} V^{2+2\alpha}} 
        \le \frac{6m}{\abs{h}^{1+2\alpha} V}.
    \]
    Since $V\ge\frac m2$ and we assume that $m\le\frac R2$ (see the start of Subsection \ref{S2.3}), we obtain 
\[
        \abs{J_{2}} 
        \le  2\int_{0}^{m/2}\frac{ 12}{\abs{h}^{1+2\alpha}} 2A\abs{h}^{1+\gamma}\,dh  
        + 2\int_{m/2}^{R}\frac{ 6m}{\abs{h}^{2+2\alpha}} 2A\abs{h}^{1+\gamma} \,dh. 
 \]
This is less than $Cm$ if $\gamma>2\alpha$ and less than $Cm(1+\ln_- m)$ if  $\gamma=2\alpha$ (for some $C$).
\end{proof}

To estimate $J_{1}$, it suffices to assume that $f(0)\notin[0,m]$.  Indeed, if $f(0)\in\{0,m\}$, then the graph of $f$ contains either $(0,0)$ or $(0,m)$, so (1) and (2) above \eqref{2.5} imply $f'(0)=0$ and therefore $J_{1}=0$.  And if $f(0)\in (0,m)$, then the definition of $m$ shows that there must be $\eta\in\bbT$ with $|\eta-\eta_t|<\delta$ such that $z(\eta)=(0,f(0))$. In that case \eqref{2.2a} yields $|\eta-\eta_t|\le\frac{4R}{M'}$ and hence for all $\xi$ between $\eta$ and $\eta_t$ we  have
\[
\left| \partial_\xi z(\xi) - \partial_\xi z(\eta_t) \right| \le M \left( \frac{4R}{M'} \right)^\gamma \le M' \le \left| \partial_\xi z(\eta_t)\right|
\]
by Lemma \ref{L2.3}.  This shows that $\partial_\xi z(\xi)\cdot \partial_\xi z(\eta_t)\ge 0$ for all these $\xi$, which clearly contradicts $(z(\eta)-z(\eta_t)) \cdot  \partial_\xi z(\eta_t)=0$.  So $f(0)\notin[0,m]$, and we define $a:=-f(0)>0$ when $f(0)<0$, and $a:=f(0)-m>0$ when $f(0)>m$. 
In both cases Lemma \ref{L2.3} yields 
\beq\lb{2.10}
\abs{f'(0)} \le B a^{\frac{\gamma}{1+\gamma}} \leq B R^{\frac{\gamma}{1+\gamma}}\le \frac 12.
\eeq

\begin{lemma} \lb{L2.5}
    We have $\abs{J_{1}}\le C a^{\frac{\gamma}{1+\gamma}}(a+m)^{-2\alpha}m$ for some constant $C$.
\end{lemma}

\begin{proof}
    Note that the definition of $m$ shows that $a\ge m$, so we could replace $a+m$ by $a$.  We will not use this so that this result also applies in Section \ref{S3}. We will assume $f(0)<0$ since the proof for the other case is virtually identical. We can then rewrite $J_{1}$ as
    \[
        J_{1} = \int_{-R}^{R}\int_{0}^{hf'(0)}
        \left( \frac{h}{(h^{2}+(v-a)^{2})^{1+\alpha}} - \frac{h}{(h^{2}+(v-a-m)^{2})^{1+\alpha}}
        \right) dv\,dh.
    \]
    We split the integral into two parts:
    \begin{align*}
        J_{3} &\coloneqq \int_{\abs{h} < a+m}\int_{0}^{hf'(0)}
         \left(\frac{h}{(h^{2}+(v-a)^{2})^{1+\alpha}} - \frac{h}{(h^{2}+(v-a-m)^{2})^{1+\alpha}}
         \right) dv\,dh, \\
        J_{4} &\coloneqq \int_{a+m\le \abs{h}\le R}\int_{0}^{hf'(0)}
         \left(\frac{h}{(h^{2}+(v-a)^{2})^{1+\alpha}} - \frac{h}{(h^{2}+(v-a-m)^{2})^{1+\alpha}}
         \right) dv\,dh.
    \end{align*}
    
    For $J_{3}$, note that for any $v$ in the domain of integration, we have
    \[
       v-a-m \leq  \abs{(a+m)f'(0)} - a - m  \leq - \frac{a+m}{2}.
    \]
    This also shows that $v-a \leq \frac{m-a}{2}$, so $\abs{v-a}\leq\abs{v-a-m}$.  The mean value theorem then gives for some $\bar{v}\in[v-a-m,v-a]$,
    \begin{align*}
        \abs{ \frac{h}{(h^{2}+(v-a)^{2})^{1+\alpha}} - \frac{h}{(h^{2} + (v-a-m)^{2})^{1+\alpha}} }
       &  = \frac{(2+2\alpha) \abs{h} (h^{2}+\bar{v}^{2})^{\alpha} \abs{\bar{v}} }
        {(h^{2}+(v-a)^{2})^{1+\alpha}(h^{2}+(v-a-m)^{2})^{1+\alpha}} \, m
     \\   &\leq \frac{6m}{\abs{h}^{1+2\alpha}(a+m)}
    \end{align*}
    because $\max\{|\bar v|,\frac{a+m}2\}\le|v-a-m|$. 
    This and \eqref{2.10} yield
    \begin{align*}
        \abs{J_{3}} \le
        \int_{-a-m}^{a+m}   \frac{6m}{\abs{h}^{1+2\alpha}(a+m)} \,\abs{hf'(0)} \,dh
        \leq \frac{12m}{1-2\alpha} \frac{\abs{f'(0)}}{(a+m)^{2\alpha}} 
        \leq \frac{12B}{1-2\alpha}a^{\frac{\gamma}{1+\gamma}} \frac m{(a+m)^{2\alpha}}.
    \end{align*}
    
    As for $J_{4}$,
    the mean value theorem yields
    \begin{align*}
        \abs{ \frac{h}{(h^{2}+(v-a)^{2})^{1+\alpha}} - \frac{h}{(h^{2} + (v-a-m)^{2})^{1+\alpha}} }
        \leq \frac{3m}{\abs{h}^{2+2\alpha}}.
    \end{align*}
    From this and \eqref{2.10} we obtain
    \begin{align*}
        \abs{J_{4}} &\le
        2 \int_{a+m}^R  \frac{3m}{\abs{h}^{2+2\alpha}} \abs{hf'(0)} \,dh
        \le  \frac{3m }{\alpha } \frac{\abs{f'(0)}}{(a+m)^{2\alpha}}
        \le \frac{3B}{\alpha} a^{\frac{\gamma}{1+\gamma}} \frac m{(a+m)^{2\alpha}}.
    \end{align*}
    This finishes the proof.
\end{proof}

The last two lemmas, together with the estimate $|I'|\le Cm$ above and $k\le K$, show that \eqref{2.5} holds when $\gamma>2\alpha$ and $\frac{\gamma}{1+\gamma}\geq 2\alpha$.   Since $\gamma =\frac{2\alpha}{1-2\alpha}$ satisfies this (and \eqref{2.0} holds for it by the hypothesis), the proof of the single-patch case of Theorem \ref{T1.3} (and of Remark 1 after it) is finished.

\subsection{Absence of Simple Splashes for all $\alpha\in(0,1)$} \lb{S2.5}

Let us now assume that only simple splashes can happen for a $C^{1,\gamma}$ patch and $\alpha\in(0,\frac 12)$.  That is,  there is $R>0$ such that for all $t$ close enough to $T$ and any $\xi_t,\eta_t\in\bbT$ satisfying $|z(\xi_t,t)-z(\eta_t,t)|=m(t)$, there is no $\xi\in\bbT$ such that $\min\{|\xi-\xi_t|,|\xi-\eta_t|\}\ge\delta$ and also $|z(\xi,t) - z(\eta_t,t)|\le R$.  This essentially means that any potential splash only involves two folds of $\partial\Omega$, although this requirement is in fact weaker than that:  multiple folds are allowed but not near minimizers of \eqref{2.2b}.  
Then in Lemma \ref{L2.5} we have $f'(0)=0$ and so $J_1=0$.  Hence Lemma \ref{L2.4} shows that a simple splash cannot occur by time $T$ if $\sup_{t\in[0,T)}\norm{\Omega(t)}_{C^{1,\gamma}}<\infty$ for some $\gamma\in[2\alpha,1]$, and $m(t)$ can decrease at most exponentially when  $\gamma>2\alpha$ and at most double-exponentially when  $\gamma=2\alpha$.  



In fact, one can even extend this result to all $\alpha\in[\frac 12,1)$.
In this case one must replace $u$ in \eqref{1.3a} (which becomes infinite on $\partial\Omega(t)$) by its normal ``component''
\[
        u_n(x,t) \coloneqq {\rm p.v.}\  \int_{\Omega(t)} c_\alpha \frac{(x-y)^{\perp}\cdot n_{x,t}}{\abs{x-y}^{2+2\alpha}} \, dy \,\, n_{x,t}
\]
(which is finite),
with $n_{x,t}$ the unit outer normal vector to $\Omega(t)$ at $x\in\partial\Omega(t)$ (see also Remark~2 after Definition 1.2 in \cite{KisYaoZla}, or \cite{KisLuo}).  If we now assume $\sup_{t\in[0,T)}\norm{\Omega(t)}_{C^{2,\gamma}}<\infty$, one can use (2) above \eqref{2.5} to show that $|f(h)-f(0)- \frac{f''(0)}2 h^2|\le A\abs{h}^{2+\gamma}$ in \eqref{2.7a}.
Then oddness-in-$h$ of the integrand will yield the estimate
\[
        \abs{J_{2}} 
        \le  2\int_{0}^{m/2}\frac{ 16}{\abs{h}^{1+2\alpha}}  2A\abs{h}^{2+\gamma} \,dh  
        + 2\int_{m/2}^{R}\frac{ 8m}{\abs{h}^{2+2\alpha}} 2A\abs{h}^{2+\gamma}\,dh 
 \]
in the proof of Lemma \ref{L2.4} whenever $\gamma\in[2\alpha-1,1]$.  Hence no finite time simple splash can happen by time $T$ in this case either, and we again obtain an exponential resp.~double-exponential lower bound on $m(t)$  when  $\gamma>2\alpha-1$ resp.~$\gamma=2\alpha-1$
 (note also that for $\alpha=\frac 12$ it even  suffices to assume $\sup_{t\in[0,T)}\norm{\Omega(t)}_{C^{1,1}}<\infty$, with $2+\gamma$ replaced by 2 above and with a double-exponential lower bound on $m(t)$).

\subsection{The Multiple Patches Case} \lb{S2.6}

In the general multiple patches case, \eqref{2.1} becomes
\[
    \sup_{t\in[0,T)} \,\sup_{(n,\xi),(j,\eta)\in Z_N \times \bbT  \,\&\, (n,\xi)\neq(j,\eta)} \frac{|n-j| + \abs{\xi-\eta}}{\abs{z_n(\xi,t) - z_j(\eta,t)}} < \infty,
\]
where $Z_N:=\{1,\dots,N\}$ and $z_n(\cdot, t)$ is a constant-speed parametrization of $\partial\Omega_n(t)$.  We choose the same $\delta$ (with all $z_n$ included in the definitions of $M,M'$), and then
\beq\lb{4.1}
    m(t)\coloneqq \min_{(n,\xi),(j,\eta)\in Z_N \times \bbT \,\&\, \abs{(n,\xi)-(j,\eta)}\geq\delta}\abs{z_n(\xi,t) - z_j(\eta,t)} \ge 0.
\eeq
The points $\xi_t$ and $\eta_t$ may now be on the boundaries of distinct patches, but that does not change our analysis, which only deals with the individual patch segments in a small rectangle centered at $\eta_t$.  The geometric lemmas are unchanged; the estimates on integrals $I'$ and $I_i$ in Subsection \ref{S2.3} only change by the factor $|\theta_1|+\dots+|\theta_N|$, hence so does the rest of the argument.  This finishes the proof of Theorem \ref{T1.3} (and of Remark 1 after it) as stated.  The claim in Remark 2 after Theorem \ref{T1.3} also extends easily to this case.

\section{Proof of Theorem \ref{T1.4}}
\label{S3}

Let us now turn to the half-plane case $D\coloneqq\mathbb{R}\times\mathbb{R}^{+}$,  when the proof is essentially identical to Theorem~\ref{T1.3} (and the $H^3$ resp.~$H^2$ local well-posedness results from  \cite{KisYaoZla, GanPat} require $\alpha\in(0,\frac 1{24})$ resp.~$\alpha\in(0,\frac 1{6})$).

Let us first recall the definition of patch solutions in this setting from \cite{KisYaoZla}.
Equation \eqref{1.1} is unchanged, and $\Delta$ in \eqref{1.1a}  is the Dirichlet Laplacian on $D$.  If we assume that $\alpha\in\left(0,\frac{1}{2}\right)$, this means that for an appropriate constant $c_{\alpha}>0$ we have
\begin{equation}\lb{3.1}
    u(x,t) = c_{\alpha}\int_{D}\left( \frac{(x-y)^{\perp}}{\abs{x-y}^{2+2\alpha}}
    - \frac{(x-\bar{y})^{\perp}}{\abs{x-\bar{y}}^{2+2\alpha}} \right)
    \omega(y,t)\,dy
\end{equation}
for each $x\in \overline{D}$, where $\bar{y}\coloneqq(y_{1},-y_{2})$. Definition~\ref{D1.2} is as before, but with the patches $\Omega_{1}(t), \dots ,\Omega_{N}(t)$ now contained in $D$ instead of $\mathbb{R}^{2}$, and with $u$ from \eqref{3.1} instead of \eqref{1.2}.  

We define $M,M',\delta$ as before and $m(t)$ via \eqref{4.1}. We also consider the reflected patches $\overline{\Omega}_{n}(t)\coloneqq \{y\in\bbR^2\setminus D \,|\,\bar{y}\in\Omega_{n}(t) \}$, which allow us to write (after dropping $c_\alpha$ via rescaling)
\[
    u(x,t) = \sum_{n=1}^{N}\theta_{n}    \int_{\Omega_{n}(t)}  \frac{(x-y)^{\perp}}{\abs{x-y}^{2+2\alpha}} \,dy \, - \,
    \sum_{n=1}^{N}\theta_{n}    \int_{\overline\Omega_{n}(t)} \frac{(x-y)^{\perp}}{\abs{x-y}^{2+2\alpha}}   \,dy.
\]
Theorem~\ref{T1.4} will now follow once we show \eqref{2.4} with this $u$.  This is proved in the same way as on $\bbR^2$, but now the boundary segments defining functions $f_i$ in Subsection \ref{S2.3} can belong to both the original and the reflected patches.  Note that the distance of $\partial \Omega_{n}(t)$ and $\partial \overline{\Omega}_{n}(t)$ can  be less than $m(t)$ (even 0 because they can touch at $\partial D$, in which case their normal vectors coincide at any point of touch).  But they obviously cannot cross (this is why we did not assume that $a\ge m$ in Lemma \ref{L2.5}), which allows us to use the same estimates as in Section \ref{S2}, modulo a factor of 2 due to the number of patches now being doubled. 

\end{document}